\documentclass[final,10pt]{amsart}
\usepackage{amsmath,amsfonts,amssymb,amscd,verbatim,mathtools}
\usepackage[margin=0.4in]{geometry}
\usepackage{fullpage}

\usepackage{xcolor}
\usepackage{multicol}
\usepackage{enumerate}
\usepackage{hyperref}

\usepackage{enumitem}
\usepackage{mathabx}
\usepackage{array}
\usepackage{tikz,tikz-cd}

\numberwithin{equation}{section}
\theoremstyle{plain}
 
\newtheorem{theorem}[equation]{Theorem} 

\newtheorem{corollary}[equation]{Corollary} 
\newtheorem{lemma}[equation]{Lemma}
\newtheorem{proposition}[equation]{Proposition}

\theoremstyle{definition}
\newtheorem{definition}[equation]{Definition}

\newtheorem{hypothesis}[equation]{Hypothesis}

\newtheorem{thmintro}{Theorem}

\newtheorem*{theorem*}{Theorem}

\DeclareMathOperator\Ext{Ext}
\DeclareMathOperator\GL{GL}

\newcommand\kk{\Bbbk}
\newcommand\inv{^{-1}}
\newcommand\iso{\cong}

\newcommand\bq{\mathbf q}
\newcommand\bp{\mathbf p}
\newcommand\br{\mathbf r}
\renewcommand\mod{~\mathrm{mod}~}

\newcommand\cA{\mathcal A}
\newcommand\cB{\mathcal B}
\newcommand\cJ{\mathcal J}
\newcommand\cR{\mathcal R}
\newcommand\cS{\mathcal S}

\newcommand\NN{\mathbb N}

\newcommand\Aq{A_2(q)}
\newcommand\Ap{A_2(p)}
\newcommand\Aqinv{A_2(q\inv)}
\newcommand\Bp{B_2(p)}
\newcommand\Bq{B_2(q)}
\newcommand\Dq{D(q)}
\newcommand\Dp{D(p)}
\newcommand\Dqinv{D(q\inv)}
\newcommand\algJ{J}

\newenvironment{psmatrix}
  {\left(\begin{smallmatrix}}
  {\end{smallmatrix}\right)}

\title{Dimension two twisted graded Calabi--Yau algebras on two-vertex quivers}
\keywords{Graded twisted Calabi--Yau algebras, isomorphism problem, modified preprojective algebra}
\subjclass{16E65,16P90,16S38,16W50}

\author[Gaddis]{Jason Gaddis}
\author[Zazycki]{Daryl Zazycki}
\address{Department of Mathematics, Miami University, Oxford, Ohio 45056, USA} 
\email{gaddisj@miamioh.edu,zazyckdw@miamioh.edu}

\begin{document}

\begin{abstract}
We classify, up to isomorphism, twisted graded Calabi--Yau algebras of dimension two on two-vertex quivers. By work of Reyes and Rogalski, such algebras may be presented as quotients of translation quivers by mesh relations. We also consider the isomorphism problem for certain families of twisted graded Calabi--Yau algebras on larger quivers.
\end{abstract}

\maketitle

\section{Introduction}
\label{sec.intro}

Throughout, $\kk$ denotes an algebraically closed field of characteristic zero and $\kk^\times$ its group of units. All algebras may be assumed to be $\kk$-algebras. For a quiver $Q$, $\kk Q$ denotes its path algebra over $\kk$. 

Given a family of algebras $\cR = \{R_\gamma\}$, the \emph{isomorphism problem} asks for necessary and sufficient conditions for an isomorphism $R_\gamma \iso R_{\gamma'}$ between two elements of $\cR$. For example, if $\cR$ is the class of Artin--Schelter regular algebras of global dimension two, then up to isomorphism, $R \in \cR$ is one of the following:
\begin{align}\label{eq.jplane}
    \text{(Jordan plane)} \qquad \kk_J[x,y] &:= \kk\langle x,y\rangle/(xy-yx+x^2), \\
    \label{eq.qplane}
    \text{(quantum plane)} \qquad \kk_q[x,y] &:= \kk\langle x,y\rangle/(xy-qyx), \qquad q \in \kk^\times.
\end{align}
No quantum plane is isomorphic to the Jordan plane, and for 
$p,q \in \kk^\times$, $\kk_q[x,y] \iso \kk_p[x,y]$ if and only if $p=q^{\pm 1}$. 

This paper is concerned with a related isomorphism problem: the classification of twisted graded Calabi--Yau algebras of dimension two of the form $\kk Q/I$, where $Q$ is a (finite) quiver with two vertices and $I$ is a homogeneous ideal generated by quadratic elements. By work of Reyes and Rogalski, it is known that such algebras have a specific form $\cA(Q,\tau)$ where $Q$ is a translation quiver and $\tau$ is an injective linear map on $\kk Q_1$ \cite{RR1}. See Section \ref{sec.tgcy} for background on algebras of this form. These algebras are shown to fit into four families.

A key advance in the isomorphism problem for graded rings was a result of Bell and Zhang \cite{BZ1}. In that work, the authors proved that two connected $\NN$-graded $\kk$-algebras generated in degree one are isomorphic if and only if there is a graded isomorphism between them. An adaptation of their result to path algebras modulo homogeneous relations \cite{gadiso} will be applied here.

\subsection{The algebras and their isomorphisms}

The preprojective algebras of type $A$ are well-known Calabi--Yau algebras. The algebras below may be considered as twisted versions of these. In \cite{EE,kleiner}, they are referred to as \emph{modified preprojective algebras}.

\begin{definition}\label{defn.An}
(1) Let $Q$ be the quiver
\begin{equation}\label{eq.An}
\begin{tikzcd}
    &   &   & e_0 \arrow[llld, "a_0"] \arrow[rrrd, "a_{n-1}^*", shift left=2] 	&           &		& \\
e_1 \arrow[rr, "a_1"] \arrow[rrru, "a_0^*", shift left=2] 	&        
&  e_2 \arrow[r, "a_2"]	\arrow[ll, "a_1^*", shift left=2]		    & \cdots   \arrow[r, "a_{n-3}"]  \arrow[l, "a_2^*", shift left=2] &	e_{n-2} \arrow[rr, "a_{n-2}"]	\arrow[l, "a_{n-3}^*", shift left=2] &           	& e_{n-1} \arrow[lllu, "a_{n-1}"] \arrow[ll, "a_{n-2}^*", shift left=2]
\end{tikzcd}
\end{equation}
For $\bq = (q_0,\hdots,q_{n-1}) \in (\kk^\times)^n$, define
$A_n(\bq) := \kk Q/ (a_ia_i^* - q_i a_{i-1}^*a_{i-1} \quad \text{for $i = 0\hdots, n-1$})$.

(2) For $n \geq 2$, let $Q$ be the quiver
\begin{equation}\label{eq.Bn}
\begin{tikzcd}
    &   & e_0 \arrow[out=120,in=60,loop,"b_0"]\arrow[lld,"a_0"']   &   &   \\
e_1 \arrow[r,"a_1"'] \arrow[out=240,in=300,loop,swap,"b_1"] & 
e_2 \arrow[r,"a_2"'] \arrow[out=240,in=300,loop,swap,"b_2"] & 
\cdots \arrow[r,"a_{n-3}"'] & 
e_{n-2} \arrow[r,"a_{n-2}"'] \arrow[out=240,in=300,loop,swap,"b_{n-2}"] & e_{n-1} \arrow[llu,"a_{n-1}"'] \arrow[out=240,in=300,loop,swap,"b_{n-1}"]
\end{tikzcd}
\end{equation}
For $\bq = (q_0,\hdots,q_{n-1}) \in (\kk^\times)^n$,
define $B_n(\bq) = \kk Q/(b_ia_i - q_i a_i b_{i+1} \quad 
    \text{for $i = 0\hdots, n-1$})$.
\end{definition}

The algebras $A_n(\bq)$ and $B_n(\bq)$ are twisted graded Calabi--Yau (see Lemma \ref{lem.AnCY} and Lemma \ref{lem.BnCY}, respectively). The algebras $A_n(\bq)$ with $q_i=1$ for all $i$ are precisely the preprojective algebras of type $A$. When $n=1$, then $A_1(\bq) \iso \kk_{q_0}[x,y]$.

\begin{thmintro}[Theorems \ref{thm.An} and \ref{thm.Bn}]
\label{thm.iso-intro}
(1) Let $n \geq 3$ and $\bq,\bp \in (\kk^\times)^n$. 
Then $A_n(\bq)\iso A_n(\bp)$ if and only if there exists $\alpha_i \in \kk^\times$ and $0 \leq k < n$ such that 
\[
\text{(1)}~p_i = \frac{\alpha_{i+k-1}}{\alpha_{i+k}} q_{i+k} 
    \quad\text{for all $0\leq i < n$, or} \qquad
\text{(2)}~p_i = \frac{\alpha_{n-i-k}}{\alpha_{n-i-k-1}} q_{n-i-k}\inv
    \quad\text{for all $0\leq i < n$.}
\]

(2) Let $n \geq 2$ and $\bq,\bp \in (\kk^\times)^n$. Then $B_n(\bq) \iso B_n(\bp)$ if and only if there exists $\alpha_i \in \kk^\times$ and $0 \leq k < n$ such that $p_i = \frac{\alpha_{i+k+1}}{\alpha_{i+k}} q_{i+k}$ for all $0\leq i < n$.
\end{thmintro}

The quiver \eqref{eq.Bn} is \emph{schurian} (for vertices $i,j$, there is at most one arrow $i\to j$) and the quiver \eqref{eq.An} is schurian for $n \geq 3$. When $n=2$, the quiver \eqref{eq.An} is not schurian and the isomorphism problem is more involved. When $n=2$, both quivers support an additional type of algebra defined below.

\begin{definition}\label{defn.twovert}
(1) Let $Q$ be the quiver
\begin{equation}\label{eq.A2}
\begin{tikzcd}
e_1 \ar[rr,bend left,swap,"a"] \ar[rr,bend left,shift left=1ex,"c"] & & e_2 \ar[ll,bend left,swap,"b"] \ar[ll,bend left,shift left=1ex,"d"]
\end{tikzcd}
\end{equation}
For $q \in \kk^\times$, define
$\Aq = \kk Q/(ab-cd, ba-qdc)$ and 
$\Dq = \kk Q/(ab-(a+c)d,ba-qdc)$.

(2) Let $Q$ be the quiver
\begin{equation}\label{eq.B2}
\begin{tikzcd}
e_1 \ar[loop,in=210,out=150,looseness=4,swap,"a"] \ar[rr,bend left,"b"] & & 
e_2 \ar[loop,in=30,out=330,looseness=4,swap, "c"] \ar[ll,bend left,"d"]
\end{tikzcd}
\end{equation}
For $q \in \kk^\times$, define
$\Bq = \kk Q/(ab-bc, cd-qda)$ and 
$\algJ = \kk Q/(a^2-bd, c^2-db)$.
\end{definition}

The cases $\algJ$ and $\Dq$ are called \emph{exceptional} because they do not (obviously) belong to families on $n$-vertices as in the case of $\Aq$ and $\Bq$. The algebras $\algJ$ and $\Dq$ are shown to be twisted Calabi--Yau in Lemmas \ref{lem.CCY} and \ref{lem.DCY}, respectively. 

\begin{thmintro}[Theorem \ref{thm.class}]
\label{thm.2vert-intro}
For $q \in \kk^\times$, the algebras $\Aq$, $\Bq$, $\algJ$, and $\Dq$ are twisted graded Calabi--Yau. If $\cA=\cA(Q,\tau)$ is a twisted graded Calabi--Yau algebra of dimension two and finite Gelfand--Kirillov dimension with $Q$ strongly connected and $|Q_0|=2$, then $\cA$ is isomorphic to one of the $\Aq$, $\Bq$, $\algJ$, or $\Dq$. Moreover, these algebras are pairwise non-isomorphic except $\Aq \iso \Aqinv$ and $\Dq \iso \Dqinv$.
\end{thmintro}

This paper is laid out as follows. Section \ref{sec.background}, gives definitions for quivers, path algebras, and twisted graded Calabi--Yau algebras. Here it is established that the algebras introduced above are twisted Calabi--Yau and Theorem \ref{thm.iso-intro} is proved. Section \ref{sec.twovert} is devoted to the proof of Theorem \ref{thm.2vert-intro}.

\subsection*{Acknowledgements}
Zazycki was supported by the Undergraduate Summer Scholars program at Miami University.

\section{Background}
\label{sec.background}

An algebra $R$ is \emph{$\NN$-graded} (or just \emph{graded}) if there is a vector space decomposition $R = \bigoplus_{k \geq 0} R_k$ such that $R_k \cdot R_\ell \subset R_{k+\ell}$. If $R$ is graded, the \emph{Hilbert series} of $R$ is $h_R(t)=\sum_{k \geq 0} \dim_\kk(R_k) t^k$.

\subsection{Quivers and path algebras}
\label{sec.qpa}

A \emph{quiver} $Q=(Q_0,Q_1,s,t)$ is a directed graph consisting of a vertex set $Q_0$, an arrow set $Q_1$, and maps $s,t:Q_1 \to Q_0$ that identify for each arrow $a \in Q_1$ its \emph{source} $s(a)$ and its \emph{target} $t(a)$. Both $Q_0$ and $Q_1$ are assumed to be finite. 
The \emph{adjacency matrix} of $Q$ is $M=(m_{ij})$ where $m_{ij}$ denotes the number of arrows with source $i$ and vertex $j$. 

A \emph{path} in $Q$ is a sequence, written $p=a_0a_1\cdots a_n$, such that $s(a_i) = t(a_{i-1})$ for $i=1,\hdots,n$. The source and target maps extend to the set of paths by $s(p)=s(a_0)$ and $t(p)=t(a_n)$. The quiver $Q$ is \emph{strongly connected} if for any vertices $i,j$, there exists a path $p$ with $s(p)=i$ and $t(p)=j$.
The quiver is \emph{schurian} if for any vertices $i,j$, there is at most one arrow $a$ with $s(a)=i$ and $s(a)=j$.

The \emph{path algebra} $\kk Q$ has a $\kk$-vector space basis consisting of paths in $Q$ and multiplication of paths $p$ and $q$ is defined by the concatenation $pq$ assuming $t(p)=s(q)$, and otherwise the product is set to $0$. At each vertex $v \in Q_0$ there is a trivial loop $e_v$ which satisfies $e_v p = p$ if $s(p)=v$ and is $0$ otherwise, and $pe_v = p$ if $t(p)=v$ and $0$ otherwise.

\subsection{Graded twisted Calabi--Yau algebras}

The Calabi--Yau condition for algebras was introduced by Ginzburg \cite{ginz}. The study of this condition is intimately connected to mirror symmetry, noncommutative projective algebraic geometry, and differential operator rings.

The \emph{enveloping algebra} of an algebra $R$ is $R^e := R \otimes R^{\mathrm{op}}$. A $\kk$-central $(R,R)$-bimodule $M$ is a left $R^e$-module where the action is given by $(r \otimes s)\cdot m = rms$ for all $m \in M$ and $r,s \in R$. The symbol $\delta_{ij}$ denotes the Kronecker-delta function.

\begin{definition}\label{defn.CY}
An algebra $R$ is said to be \emph{homologically smooth (over $\kk$)} if $R$ has a finite-length resolution by finitely generated projective $R^e$-modules. The algebra $R$ is \emph{twisted Calabi--Yau} of dimension $d$ if $R$ is homologically smooth and there exists an invertible $(R,R)$-bimodule $U$ such that there are right $R^e$-module isomorphisms $\Ext_{R^e}^i(R,R^e) \iso \delta_{id} U$. If $U=R$, then $R$ is \emph{Calabi--Yau}.
\end{definition}

In this paper, $U=R$ or $U=R^\sigma$, the twist of $R$ by an automorphism $\sigma$. In this case, $\sigma$ is called the \emph{Nakayama automorphism} of $R$.

Examples of Calabi--Yau algebras include polynomial rings, the $n$th Weyl algebra, and the preprojective algebra of non-Dynkin quiver. By \cite{RRZ}, a connected $\NN$-graded algebra is twisted Calabi--Yau if and only if it is Artin--Schelter regular.

\subsection{Twisted graded Calabi--Yau algebras of dimension two}
\label{sec.tgcy}

Let $Q$ be a (finite) quiver with $|Q_0|=n$ in which all arrows are weighted with degree one. Let $V=\kk Q_1$, let $\mu$ be a permutation of $Q_0$, and let $\tau:V \to V$ be a bijective $\kk$-linear graded map such that $\tau(e_i V e_j) \subset e_{\mu\inv(j)} V e_i$ for all $i,j,d$. Set $\omega= \sum_{x \in Q_1} \tau(x)x$. Then
\[ \cA(Q,\tau) = \kk Q/(\omega) = \kk Q/(\omega_1,\hdots,\omega_n)\]
where
\[ \omega_i = \omega e_i = e_{\mu\inv(i)} \omega e_i = \sum_{x \in \cB \cap \kk Qe_i} \tau(x)x.\]

Let $M$ denote the adjacency matrix of $Q$ and $P$ the permutation matrix defined by $P_{ij}=\delta_{\mu(i)j}$. The \emph{matrix-valued Hilbert series} of $\cA=\cA(Q,\tau)$ is $h_{\cA}(t) = \sum_{k \geq 0} H_k t^k$ where $H_k \in M_n(\NN)$ and $(H_k)_{ij}$ records the number of paths (modulo relations) from vertex $i$ to vertex $j$. By \cite[Lemma 7.6]{RR1}, $\cA=\cA(Q,\tau)$ is twisted Calabi--Yau of dimension two if and only if $h_\cA(t) = (p(t))\inv$ where $p(t)=I-Mt+Pt^2$. This fact will be used frequently without further comment.

\begin{lemma}\label{lem.AnCY}
For $\bq \in (\kk^\times)^n$, the algebras $A_n(\bq)$ are twisted Calabi--Yau.
\end{lemma}
\begin{proof}
Let $Q$ denote the quiver in \eqref{eq.An}. Let $\mu=\mathrm{id}$ and define $\tau$ by $\tau(a_i) = -q_{i+1} a_i^*$ and $\tau(a_i^*) = a_i$
%\[ \tau(a_i) = -q_{i+1} a_i^* \qquad \tau(a_i^*) = a_i\]
extended linearly so that $\tau:V \to \kk Q$ is an injective $\kk$-linear graded map. It is clear that $W=\tau(V)$ is again an arrow space for $\kk Q$, so $A_n(\bq) = \cA(Q,\tau)$. Let $\omega = \sum_{x \in Q_1} \tau(x)x$ so that 
\[ \omega_i = e_i \omega e_i = \tau(a_i^*)a_i^* + \tau(a_{i-1})a_{i-1}
	= a_ia_i^* - q_i a_{i-1}^*a_{i-1}.\]
Now $h_{A_n(\bq)} = (I-Mt + It^2)\inv$ by \cite[Theorem 3.5.1]{EE}.
\end{proof}

Going forward, it will be necessary to use Diamond Lemma \cite{berg_diamond} arguments to obtain the Hilbert series of these algebras.

\begin{lemma}\label{lem.BnCY}
For $\bq \in (\kk^\times)^n$, the algebras $B_n(\bq)$ are twisted Calabi--Yau.
\end{lemma}
\begin{proof}
Let $Q$ denote the quiver in \eqref{eq.Bn}. Define $\tau$ by $\tau(a_i) = b_i$ and $\tau(b_i) = -q_{i-1} a_{i-1}$
%\[ \tau(a_i) = b_i \qquad \tau(b_i) = -q_{i-1} a_{i-1}\]
extended linearly so that $\tau:V \to \kk Q$ is an injective $\kk$-linear graded map. It is clear that $W=\tau(V)$ is again an arrow space for $\kk Q$. So, $B_n(\bq) \iso \cA(Q,\tau)$. Let $\omega = \sum_{x \in Q_1} \tau(x)x$ so that 
\[ \omega_i = e_i \omega e_{i+1} = \tau(a_i)a_i + \tau(b_{i+1})b_{i+1}
	= b_ia_i - q_ia_ib_{i+1}.\]

In this case, $\mu(e_i)=e_{i+1}$ and so $P$ is the permutation matrix with $P_{j-1,j}=1$ and $P_{ij}=0$ with $i \neq j-1$. Then the adjacency matrix of $Q$ satisfies $M=I+P$. Write $(1-Mt+Pt^2)\inv = \sum_{k=0} H_k t^k$. The coefficients satisfy the recursive formula $H_{k+1} = MH_k - PH_{k-1}$. Note that $H_0 = I$ and $H_1 = M = I+P$. Suppose $H_k = \sum_{i=0}^k P^i$, then inductively
\[ H_{k+1} 
    = M \sum_{i=0}^k P^i - P \sum_{i=0}^{k-1} P^i
    = (I+P) \sum_{i=0}^k P^i - \sum_{i=0}^{k} P^i + I
    = P \sum_{i=0}^k P^i + I
    = \sum_{i=0}^{k+1} P^i.\]

Using the relations, any path of length $k$ from vertex $i$ to vertex $j$ can be written in the form $a_i a_{i+1} \cdots a_{j-1}b_j^\ell$ where $a_i a_{i+1} \cdots a_{j-1}$ has length $d$ and $d+\ell=k$. Multiplying each $a_i a_{i+1} \cdots a_{j-1}b_j^\ell$ by $b_j$ on the right gives a length $k+1$ path from vertex $i$ to vertex $j$. There is one additional path, composed only of the $a_i$, if and only if $k+1 \equiv j \mod n$ if and only if $(P^{k+1})_{ij} \neq 0$. It follows that $h_{B_n(\bq)}=(I-Mt+Pt^2)\inv$ as computed above.
\end{proof}

An alternative approach to Lemmas \ref{lem.AnCY} and \ref{lem.BnCY} would be to realize the $A_n(\bq)$ and $B_n(\bq)$ are Ore extensions. The next two results establish the twisted Calabi--Yau condition for the exceptional cases.

\begin{lemma}\label{lem.CCY}
The algebra $\algJ$ is twisted Calabi--Yau.
\end{lemma}
\begin{proof}
Let $Q$ be the quiver in \eqref{eq.B2}. Then $\algJ=\cA(Q,\tau)$ where
\[ 
    \tau(a) = a, \qquad
    \tau(b) = -d, \qquad
    \tau(c) = c, \qquad
    \tau(d)= -b.
\]
Then
\[ \omega = \tau(a)a + \tau(b)b + \tau(c)c + \tau(d)d
    = a^2 - db + c^2 - bd.\]
So, $\mu=\mathrm{id}_{Q_0}$ and $P=I$.  

Write $(I-Mt+It^2)\inv = \sum_{k=0} H_k t^k$. An easy induction shows that for all $k \in \NN$:
\[
H_k = \begin{cases}
    \begin{psmatrix}2i+1 & 2i+2 \\ 2i+2 & 2i+1\end{psmatrix}
        & \text{ if $k=4i+2$} \\
    \begin{psmatrix}2i+1 & 2i \\ 2i & 2i+1\end{psmatrix}
        & \text{ if $k=4i$} \\
    \begin{psmatrix}i+1 & i+1 \\ i+1 & i+1\end{psmatrix}
        & \text{ if $k=2i+1$}.
\end{cases}
\]

Consider the ordering of variables $d>c>b>a$. Hence, the leading terms of the relations are $bd$ and $db$. There are two overlaps to resolve:
\begin{align*}
    a^2b &= (bd)b = b(db) = bc^2, \\
    c^2d &= (db)d = d(bd) = da^2.
\end{align*}
This implies that there are additional relations $a^2b-bc^2$ and $c^2d-da^2$ (with leading terms $bc^2$ and $da^2$, respectively). The remaining ambiguities $dbc^2$ and $bda^2$ resolve.

Suppose a path begins at $e_1$. By symmetry, the argument is identical for paths that begin at $e_2$. Then every path is has one of the following forms:
\[ a^\ell (bcda)^{\ell'}, \quad
   a^\ell (bcda)^{\ell'} bcd, \quad 
   a^\ell (bcda)^{\ell'} b,  \quad
   a^\ell (bcda)^{\ell'} bc.\]
Counting these paths will establish that $h_\algJ = (I-Mt+It^2)\inv$ as computed above. One example of how to do this is below and the remainder are left to the reader.

Suppose a path $p$ of length $4i$ has source and target $e_1$. Then that path has either the form $a^\ell (bcda)^{\ell'}$ with $\ell+4\ell' = k$ or of the form $a^\ell (bcda)^{\ell'} bcd$ where $\ell+4\ell' + 3=k$. If $k=4i$ or $k=4i+1$ for some $i$, then there are exactly $i+1$ of the first type and $i$ of the second type. 

On the other hand, if $k=2i+1$ for some $i$, then the argument depends on the parity of $i$. If $i$ is even, then there are $(i/2)+1$ paths of the first type and $i/2$ paths of the second type. If $i$ is odd, then there are $(i+1)/2$ paths of the first type and $(i+1)/2$ paths of the second type.
\end{proof}

%Now it follows from \cite[Lemma 7.6]{RR1} that $\algJ$ is twisted Calabi--Yau.

\begin{lemma}\label{lem.DCY}
Let $q \in \kk^\times$. The algebra $\Dq$ is twisted Calabi--Yau.
\end{lemma}
\begin{proof}
Let $Q$ be the quiver in \eqref{eq.A2}.
Then $\Dq = \cA(Q,\tau)$ where
\[ 
    \tau(a)=b, \qquad
    \tau(b)=a, \qquad
    \tau(c)=-qd, \qquad
    \tau(d)=-(a+c).
\]
Then
\[ \omega = \tau(a)a + \tau(b)b + \tau(c)c + \tau(d)d
    = ba + ab - qdc - (a+c)d.\]
So, $\mu=\mathrm{id}_{Q_0}$ and $P=I$. 

Write $(I-Mt+
It^2)\inv=\sum H_k t^k$. 
Inductively, it is easy to show that 
\[
H_k = \begin{cases}
    \begin{psmatrix}k+1 & 0 \\ 0 & k+1\end{psmatrix}
         & \text{ if $k$ is even} \\
    \begin{psmatrix}0 & k+1 \\ k+1 & 0\end{psmatrix}
         & \text{ if $k$ is odd.}
\end{cases}
\]

Consider the ordering of variables $c>b>d>a$. Then the leading terms of the relations are $ba$ and $cd$. There are no overlaps to check. Let $p$ be a path of length $2\ell$ from $e_1$ to $e_1$. Then $p$ has the form $(ad)^i (cb)^{\ell-i}$ for $i=0,\hdots,\ell$ or $(ad)^i(ab)(cb)^{\ell-i-1}$ for $i=0,\hdots,\ell-1$. It follows that there are $2\ell+1$ such paths. The count for even length paths from $e_2$ to $e_2$ is the same.

Now suppose $p$ is a path of length $2\ell + 1$ from $e_1$ to $e_2$. Then $p$ has the form $(ad)^\ell a$, $(ad)^i (cb)^{\ell-i}c$ for $i=0,\hdots,\ell$, or $(ad)^i(ab)(cb)^{\ell-i-1}c$ for $i=0,\hdots,\ell-1$. This gives a total of $2\ell+2$ such paths. The count for odd length paths from $e_2$ to $e_1$ is the same. Thus, $\Dq$ has the correct Hilbert series.
\end{proof}

%By \cite[Lemma 7.6]{RR1}, $\Dq$ is twisted Calabi--Yau.
 
\subsection{Graded vs ungraded isomorphisms}

%A theorem of Bell and Zhang proves that if two connected graded algebras finitely generated in degree one are isomorphic as ungraded algebras, then they are isomorphic as graded algebras \cite[Theorem 1]{BZ1}. This was extended in \cite{gadiso} to the non-connected setting. Some notation is needed first to set up the statement of this theorem.

Let $Q,Q'$ be finite quivers with $|Q_0|=|Q_0'|$. An element $\sum p_k \in Q$ is \emph{homogeneous} if all summands $p_k,p_\ell$ have the same path length, $s(p_k)=s(p_\ell)$, and $t(p_k)=t(p_\ell)$. An ideal $\cJ$ in $\kk Q$ is homogeneous if it is generated by homogeneous elements. Let $A = \kk Q/\cJ$ and $B=\kk Q'/\cJ'$ with $\cJ,\cJ'$ homogeneous ideals in degree at least two. Then $A$ and $B$ are \emph{isomorphic as graded path algebras} if $M_{Q'} = PM_Q P\inv$ for some permutation matrix $P$ corresponding to $\sigma \in \cS_{|Q_0|}$ and there exists an algebra isomorphism $\phi:A \to B$ such that $\phi( (A_k)_{uv} ) = (B_k)_{\sigma(u)\sigma(v)}$.

\begin{theorem}[{\cite[Theorem 8]{gadiso}}]\label{thm.graded}
Let $A = \kk Q/\cJ$ and $B=\kk Q'/\cJ'$ with $\cJ,\cJ'$ homogeneous ideals in degree at least two. If $A \iso B$ as (ungraded) algebras, then $A \iso B$ as graded path algebras.
\end{theorem}

If $Q$ and $Q'$ are quivers on a single vertex, then Theorem \ref{thm.graded} follows from \cite[Theorem 1]{BZ1}. We now apply this theorem to study the isomorphism problem for the $A_n(\bq)$ and the $B_n(\bq)$.

\subsection{The case \texorpdfstring{$A_n(\bq)$}{Type A}}
The first result in this section establishes certain important isomorphisms amongst the $A_n(\bq)$.

\begin{lemma}\label{lem.An}
Let $n \geq 3$ and $\bq \in (\kk^\times)^n$. 

%%Scalar isomorphisms
(1) Set $p_i=\frac{\alpha_{i-1}}{\alpha_i}q_i$ for all $0\leq i < n$. 
The linear map $\phi_\alpha:A_n(\bq)\to A_n(\bp)$ defined by
\[
    \phi_\alpha(e_i) = e_i, \qquad
    \phi_\alpha(a_i) = \alpha_ia_i, \qquad
    \phi_\alpha(a_i^*) = a_i^*,
\]
extends to an isomorphism.

%%Rotation isomorphisms
(2) Set $p_i=q_{i-1}$ for all $0\leq i < n$. 
The linear map $\psi:A_n(\bq) \to A_n(\bp)$ defined by
\[
    \psi(e_i) = e_{i+1}, \qquad
    \psi(a_i) = a_{i+1}, \qquad 
    \psi(a_i^*) = a_{i+1}^*,
\]
extends to an isomorphism.

%%Reflection isomorphisms
(3) Set $p_i=q_{n-i}\inv$ for all $0\leq i < n$. 
The linear map $\pi:A_n(\bq) \to A_n(\bp)$ defined by
\[
    \pi(e_i) = e_{n-i}, \qquad
    \pi(a_i) = a_{n-i-1}^*, \qquad 
    \pi(a_i^*) = a_{n-i-1},
\]
extends to an isomorphism.
\end{lemma}
\begin{proof}
The maps are all clearly bijective. It suffices to show that each respects the relations on $A_n(\bq)$. For $\phi_\alpha$:
\[
\phi_\alpha(a_i)\phi_\alpha(a_i^*)-q_i\phi_\alpha(a_{i-1}^*)\phi_\alpha(a_{i-1})
    = \alpha_ia_ia_i^*-q_i\alpha_{i-1}a_{i-1}^*a_{i-1}
%    = \alpha_i\left(a_ia_{i}^*-q_i\frac{\alpha_{i-1}}{\alpha_{i}}a_{i-1}^*a_{i-1}\right)=0.
    = \alpha_i\left(a_ia_{i}^*-p_i a_{i-1}^*a_{i-1}\right)=0.
\]
For $\psi$:
\[
\psi(a_i)\psi(a_i^*)-q_i\psi(a_{i-1}^*)\psi(a_{i-1})
    = a_{i+1}a_{i+1}^*-q_ia_i^*a_i
    = a_{i+1}a_{i+1}^*-p_{i+1}a_i^*a_i=0.
\]
Finally, for $\pi$:
\begin{align*}
\pi(a_i)\pi(a_i^*)-q_i\pi(a_{i-1}^*)\pi(a_{i-1}) 
        &= a_{n-i-1}^*a_{n-i-1}-q_ia_{n-i}a_{n-i}^* \\
        &= q_i(q_i\inv a_{n-i-1}^*a_{n-i-1}-a_{n-i}a_{n-i}^*) \\
        &= q_i(p_{n-i} a_{n-i-1}^*a_{n-i-1}-a_{n-i}a_{n-i}^*)
        =0.
\end{align*}
Hence, each map extends to an isomorphism.
\end{proof}

\begin{theorem}\label{thm.An}
Let $n \geq 3$ and $\bq,\bp \in (\kk^\times)^n$. Then $A_n(\bq)\iso A_n(\bp)$ if and only if there exists $\alpha_i \in \kk^\times$ and $0 \leq k < n$ such that 
\[
\text{(1)}~p_i = \frac{\alpha_{i+k-1}}{\alpha_{i+k}} q_{i+k} 
    \quad\text{for all $0\leq i < n$, or} \qquad
\text{(2)}~p_i = \frac{\alpha_{n-i-k}}{\alpha_{n-i-k-1}} q_{n-i-k}\inv
    \quad\text{for all $0\leq i < n$.}
\]
\end{theorem}
\begin{proof}
If (1) holds for some $\alpha_i \in \kk^\times$, then by Lemma \ref{lem.An}, $\psi^{-k} \circ \phi_\alpha$ gives the intended isomorphism. If (2) holds for some $\alpha_i \in \kk^\times$, then by Lemma \ref{lem.An}, $\psi^{-k} \circ \pi \circ \phi_\alpha$ gives the intended isomorphism. 

Assume $\Phi:A_n(\bq)\to A_n(\bp)$ is a (graded) isomorphism. Then $\Phi$ restricts to a permutation of the set of vertices $\{e_i\}$. Since $a_i$ is an arrow from $e_i$ to $e_{i+1}$, then this permutation preserves adjacency. Hence, $\Phi$ may be regarded as an element of the dihedral group on $n$ vertices.

First, suppose that $\Phi(e_i)=e_i$ for all $i$. Because $A_n(\bq)$ is schurian, then there are scalars $\beta_i,\beta_i^* \in \kk^\times$ such that $\Phi(a_i)=\beta_i a_i$ and $\Phi(a_i^*)=\beta_i^* a_i^*$. Then
\[
0 = \Phi(a_ia_i^* - q_ia_{i-1}^*a_{i-1})
    = \beta_i\beta_i^* a_ia_i^* -  q_i\beta_{i-1}\beta_{i-1}^* a_{i-1}^*a_{i-1}
    = (p_i\beta_i\beta_i^* -  q_i\beta_{i-1}\beta_{i-1}^*)a_{i-1}^*a_{i-1}.
\]
Set $\alpha_i=\beta_i\beta_i^*$. Then this gives (1) for $k=0$.

Now assume that $\Phi$ corresponds to a (nontrivial) rotation, so $\Phi(e_i)=e_{i-k}$ for some $k$, $0 < k < n$. Set $r_i = p_{i-k}$ so that $\psi^k:A_n(\bp) \to A_n(\br)$ is an isomorphism. Then $\psi^k \circ \Phi:A_n(\bp) \to A_n(\br)$ is an isomorphism that fixes the vertices. Consequently, there are scalars $\alpha_i \in \kk^\times$ such that $\frac{\alpha_{i-1}}{\alpha_i} q_i = r_i = p_{i-k}$.

Finally, assume $\Phi$ corresponds to a reflection. Similar to the previous case, there exists $k$, $0 \leq k < n$, such that $\pi \circ \psi^k \circ \Phi:A_n(\bq) \to A_n(\br)$ is an isomorphism that fixes the vertices, where $r_i = \frac{\alpha_{i-1}}{\alpha_i} q_i$. Since $\pi \circ \psi^k:A(\bp) \to A_n(\br)$, then $r_i = (p_{n-i-k})\inv$. It follows that $p_{n-i-k} = \frac{\alpha_i}{\alpha_{i-1}} q_i\inv$. So, $p_i = \frac{\alpha_{n-i-k}}{\alpha_{n-i-k-1}} q_{n-i-k}\inv$.
\end{proof}

\subsection{The case \texorpdfstring{$B_n(\bq)$}{Type B}}
The quiver \eqref{eq.Bn} is easier to to study because it only has rotational symmetry. The proofs in this case are almost identical to those above for the $A_n(\bq)$ case.

\begin{lemma}\label{lem.Bn}
Let $n \geq 2$ and $\bq \in (\kk^\times)^n$. 

%%Scalar isomorphisms
(1) Set $p_i=\frac{\alpha_{i+1}}{\alpha_i}q_i$ for all $0\leq i < n$. 
The linear map $\phi_\alpha:B_n(\bq)\to B_n(\bp)$ defined by
\[
    \phi_\alpha(e_i) = e_i, \qquad
    \phi_\alpha(a_i) = a_i, \qquad
    \phi_\alpha(b_i) = \alpha_i b_i,
\]
extends to an isomorphism.

%%Rotation isomorphisms
(2) Set $p_i=q_{i-1}$ for all $0\leq i < n$. 
The linear map $\psi:B_n(\bq) \to B_n(\bp)$ defined by
\[
    \psi(e_i) = e_{i+1}, \qquad
    \psi(a_i) = a_{i+1}, \qquad 
    \psi(b_i) = b_{i+1},
\]
extends to an isomorphism.
\end{lemma}
\begin{proof}
The maps are all clearly bijective. It suffices to show that each respects the relations on $B_n(\bq)$. For $\phi_\alpha$:
\[
\phi_\alpha(b_i)\phi_\alpha(a_i) - q_i \phi_\alpha(a_i)\phi_\alpha(b_{i+1})
    = \alpha_i b_ia_i - q_i \alpha_{i+1} a_ib_{i+1}
    = \alpha_i (b_ia_i - p_i a_ib_{i+1})
    = 0.
\]
For $\psi$:
\[
\psi(b_i)\psi(a_i) - q_i \psi(a_i)\psi(b_{i+1})
    = b_{i+1}a_{i+1} - q_i a_{i+1}b_{i+1}
    = b_{i+1}a_{i+1} - p_{i+1} a_{i+1}b_{i+1}
    = 0.
\]
Hence, each map extends to an isomorphism.
\end{proof}

\begin{theorem}\label{thm.Bn}
Let $n \geq 2$ and $\bq,\bp \in (\kk^\times)^n$. Then $B_n(\bq) \iso B_n(\bp)$ if and only if there exists $\alpha_i \in \kk^\times$ and $0 \leq k < n$ such that $p_i = \frac{\alpha_{i+k+1}}{\alpha_{i+k}} q_{i+k}$ for all $0\leq i < n$.
\end{theorem}
\begin{proof}
Suppose $p_i = \frac{\alpha_{i+k+1}}{\alpha_{i+k}} q_{i+k}$ for some $\alpha_i \in \kk^\times$. By Lemma \ref{lem.Bn}, $\psi^{-k} \circ \phi_\alpha$ gives the intended isomorphism.

Assume $\Phi:B_n(\bq)\to B_n(\bp)$ is a (graded) isomorphism. Then $\Phi$ restricts to a permutation of the set of vertices $\{e_i\}$. Since $a_i$ is an arrow from $e_i$ to $e_{i+1}$, then this permutation preserves adjacency. Hence, $\Phi$ may be regarded as an element of the group of rotations of order $n$.

First suppose that $\Phi(e_i)=e_i$ for all $i$. Because $B_n(\bq)$ is schurian, then there are scalars $\beta_i,\gamma_i \in \kk^\times$ such that $\Phi(a_i)=\gamma_i a_i$ and $\Phi(b_i)=\beta_i b_i$. Then
\[
0 = \Phi(b_ia_i - q_i a_i b_{i+1})
     = \gamma_i\beta_i b_ia_i - q_i \gamma_i\beta_{i+1} a_i b_{i+1}
     = \gamma_i (\beta_i p_i - q_i \beta_{i+1}) a_i b_{i+1}.
\]
Setting $\alpha_i=\beta_i$ gives the result for $k=0$.

Now assume that $\Phi$ corresponds to a (nontrivial) rotation, so $\Phi(e_i)=e_{i+k}$ for some $k$, $0 < k < n$. Set $r_i = p_{i-k}$ so that $\psi^k:B_n(\bp) \to B_n(\br)$ is an isomorphism. Then $\psi^k \circ \Phi:B_n(\bq) \to B_n(\br)$ is an isomorphism that fixes the vertices. Consequently, there are scalars $\alpha_i \in \kk^\times$ such that $\frac{\alpha_{i+1}}{\alpha_i} q_i = r_i = p_{i-k}$.
\end{proof}

Because the quiver \eqref{eq.Bn} in the case $n=2$ is schurian, the result for $\Bq$ is an immediate corollary of Theorem \ref{thm.Bn}.

\begin{corollary}\label{cor.B2}
Let $p,q \in \kk^\times$. If $\Bq \iso \Bp$, then $p=q$.
\end{corollary}

\section{Two-vertex quivers}\label{sec.twovert}

This section is devoted to classifying algebras satisfying the following hypotheses.

\begin{hypothesis}\label{hyp.CY2}
Assume $\cA=\cA(Q,\tau)$ is a twisted graded Calabi--Yau algebra of dimension two and finite Gelfand--Kirillov (GK) dimension. Further assume that $Q$ is strongly connected and that $|Q_0|=2$.
\end{hypothesis}

The hypothesis that $Q$ is strongly connected is to rule out the case that $A$ is the direct some of two (connected) algebras. The first result of this section establishes that algebras satisfying Hypothesis \ref{hyp.CY2} must live on one of two quivers.

\begin{lemma}\label{lem.2vert}
Assume Hypothesis \ref{hyp.CY2}. The adjacency matrix $M \in M_2(\kk)$ of $Q$ is $\begin{psmatrix}0 & 2 \\ 2 & 0\end{psmatrix}$ or $\begin{psmatrix}1 & 1 \\ 1 & 1\end{psmatrix}$.
\end{lemma}
\begin{proof}
Write $M=\left(\begin{smallmatrix}a & b \\ c & d\end{smallmatrix}\right)$ with $a,b,c,d \in \NN$. Let $P$ be the permutation matrix of $Q_0$  corresponding to the Nakayama autormorphism of $\cA$. So, either $P=I_2$ or $P = \begin{psmatrix}0 & 1 \\ 1 & 0\end{psmatrix}$.

By \cite[Theorem 7.8]{RR1}, $M$, $M^T$, and $P$ pairwise commute, and the spectral radius of $M$ is $\rho(M)=2$. Since $M$ and $M^T$ commute, then $M$ is symmetric, so $b=c$. If  $P = \begin{psmatrix}0 & 1 \\ 1 & 0\end{psmatrix}$, then $MP=PM$ implies that $a=d$. Thus, $\rho(M)=a+b$ and the result is immediate.

Suppose $P=I_2$. The eigenvalues of $M$ are 
\[ e_{\pm} = \frac{1}{2}\left((a+d) \pm \sqrt{(a-d)^2 + 4b^2}\right). \]
Since $\rho(M)=2$, then $e_+ = 2$. Thus, $a+d \leq 4$. Given a pair $(a,d)$ with $0 \leq a,d \leq 4$, $a+d\neq 4$, the equation $e_+=2$ may be solved for $b$. The only pairs giving integer solutions are $(0,0), (2,0), (1,1), (0,2), (2,2)$. With the restriction on $Q$ strongly connected, this leaves only the matrices in the statement of the lemma.
\end{proof}

\begin{lemma}\label{lem.M1}
Assume Hypothesis \ref{hyp.CY2}. Suppose the adjacency matrix of $Q$ is $M=\begin{psmatrix}0 & 2 \\ 2 & 0\end{psmatrix}$. 
Then $P=I_2$ and for some $q \in \kk^\times$,
$\cA \iso \Aq$ or $\cA \iso \Dq$.
\end{lemma}
\begin{proof}
Suppose $P = \begin{psmatrix}0 & 1 \\ 1 & 0\end{psmatrix}$. Then $p(t)=I-Mt+Pt^2$ and $\det(p(t)) = -(t^2-2t-1)(t-1)^2$. Since the roots of $(t^2-2t-1)$ do not lie on the unit circle, then $\cA$ does not have finite GK dimension \cite[Theorem 1.1]{RR1}.

Label the quiver as in \eqref{eq.A2}. Then
\[
\tau(a) = \alpha_{11} b + \alpha_{12} d, \quad
\tau(c) = \alpha_{21} b + \alpha_{22} d, \quad
\tau(b) = \beta_{11} a + \beta_{12} c, \quad
\tau(d) = \beta_{21} a + \beta_{22} c,
\]
where $\alpha=(\alpha_{ij}) \in \GL_2(\kk)$ and $\beta=(\beta_{ij}) \in \GL_2(\kk)$.

Suppose $\alpha\beta$ has two linearly independent eigenvectors $v_1,v_2$ with corresponding eigenvalues $\lambda_1,\lambda_2$. Set $u_1=\beta v_1$ and $u_2=\beta v_2$, so the $u_i$ are necessarily linearly independent and $\alpha u_i = \lambda_i v_i$.
Set
\[
\tilde{a} = (u_1)_1 b + (u_1)_2 d, \quad
\tilde{c} = (u_2)_1 b + (u_2)_2 d, \quad
\tilde{b} = (v_1)_1 a + (v_1)_2 c, \quad
\tilde{c} = (v_2)_1 a + (v_2)_2 c.
\]
Now 
\[ 
\tau(\tilde{a})=\lambda_1 \tilde{b}, \quad
\tau(\tilde{c})=\lambda_2 \tilde{d}, \quad
\tau(\tilde{b})=\tilde{a}, \quad
\tau(\tilde{d})=\tilde{c}.
\]
Using $\{\tilde{a},\tilde{b},\tilde{c},\tilde{d}\}$ as a basis for $\kk Q_1$ gives
\[ 
\omega = \lambda_1 \tilde{b}\tilde{a}
    + \tilde{a}\tilde{b} + \lambda_2 \tilde{d}\tilde{c} + \tilde{c}\tilde{d}.
\]
Replacing $\tilde{c}$ with $-\tilde{c}$ and setting $q=\lambda_2/\lambda_1$ gives $\cA \iso \Aq$.

Otherwise, up to a change of variable, there is no loss in assuming
$\alpha\beta=
\begin{psmatrix}\lambda & 1 \\ 0 & \lambda\end{psmatrix}$ for some $\lambda \in \kk^\times$. Let $v$ be an eigenvector corresponding to the eigenvalue $\lambda$, and let $u=\beta v$. As above, set
\[
\tilde{a} = u_1 b + u_2 d, \quad
\tilde{b} = v_1 a + v_2 c.
\]
So, $\tau(\tilde{a})=\lambda \tilde{b}$ and $\tau(\tilde{b})=\tilde{a}$. Then there are scalars $\gamma_1,\gamma_2$ such that  $\tau(c) = \gamma_1 \tilde{b} + \gamma_2 d$. By injectivity of $\tau$, $\gamma_2 \neq 0$. Set $\tilde{d}=\gamma_1 \tilde{b} + \gamma_2 d$, so that $\tau(c)=\tilde{d}$. Then
\[ 
\omega = \lambda \tilde{b}\tilde{a}
    + \tilde{a}\tilde{b} + \tilde{d}c + (\eta_1 \tilde{a} + \eta_2 c)\tilde{d}
\]
for some $\eta_1,\eta_2 \in \kk$. Setting $\eta_1 =0$ reduces to the case above, so assume $\eta_1 \neq 0$ and replace $\tilde{d}$ with $-\eta_1\inv\tilde{d}$ so that
\[ 
\omega = \lambda \tilde{b}\tilde{a}
    + \tilde{a}\tilde{b} - \tilde{d}c - (\tilde{a} + \eta_1\inv\eta_2 c)\tilde{d}.
\]
By injectivity, $\eta_2\neq0$. Finally, replace $c$ with $\tilde{c}=\eta_1\inv\eta_2 c$ so that
\[ 
\omega = \lambda \tilde{b}\tilde{a}
    + \tilde{a}\tilde{b} - \eta_1\inv\eta_2\tilde{d}\tilde{c} - (\tilde{a} + \tilde{c})\tilde{d}.
\]
Set $q=\eta_2(\lambda\eta_1)\inv$. Therefore, $\cA \iso \Dq$.
\end{proof}

\begin{lemma}\label{lem.M2}
Assume Hypothesis \ref{hyp.CY2}. Suppose the adjacency matrix of $Q$ is $M=\begin{psmatrix}1 & 1 \\ 1 & 1\end{psmatrix}$. If $P=I_2$, then $\cA \iso \algJ$. If $P = \begin{psmatrix}0 & 1 \\ 1 & 0\end{psmatrix}$ and $\cA \iso \Bq$ for some $q \in \kk^\times$.
\end{lemma}
\begin{proof}
If $P=I_2$, then there are nonzero scalars $\alpha_i \in \kk^\times$ such that
\[
\tau(a)=\alpha_1 a, \qquad 
\tau(b)=\alpha_2 d, \qquad
\tau(c)=\alpha_3 b, \qquad
\tau(d)=\alpha_4 c.
\]
Thus, $\omega = \alpha_1 a^2 + \alpha_2 bd + \alpha_3 c^2 + \alpha_4 bd$. The change of variable $a \mapsto (i\sqrt{\alpha_2/\alpha_1})a$ and $c \mapsto (i\sqrt{\alpha_4/\alpha_3})c$ shows that $\cA \iso \algJ$.

Now suppose $P = \begin{psmatrix}0 & 1 \\ 1 & 0\end{psmatrix}$. Then there are nonzero scalars $\beta_i \in \kk^\times$ such that
\[
    \tau(a) = \beta_1 d, \qquad
    \tau(b) = \beta_2 a, \qquad
    \tau(c) = \beta_3 b, \qquad
    \tau(d) = \beta_4 c.
\]
Thus, $\omega = \beta_1 da + \beta_2 ab + \beta_3 bc + \beta_4 cd$. Replacing $a$ with $(-\beta_3\beta_2\inv)\tilde{a}$ gives
\[ \omega = -\beta_1\beta_2\inv\beta_3 d\tilde{a} - \beta_3 \tilde{a}b + \beta_3 bc + \beta_4 cd.\]
Set $q=(\beta_1\beta_3)/(\beta_2\beta_4)$. Therefore, $\cA \iso \Bq$.
\end{proof}

It is left only to consider isomorphisms within the families $\Aq$ and $\Dq$. The family $\Bq$ was considered in Corollary \ref{cor.B2}.

Suppose $\Phi:R \to R$ is a graded isomorphism between two algebras satisfying Hypothesis \ref{hyp.CY2}. Then $\Phi$ either fixes or swaps the trivial paths $e_1$ and $e_2$. An isomorphism $\Phi$ satisfying $\Phi(e_1) = e_1$ and $\Phi(e_2) = e_2$ will be called a \emph{type I isomorphism}. An isomorphism $\Phi$ satisfying $\Phi(e_1) = e_2$ and $\Phi(e_2) = e_1$ will be called a \emph{type II isomorphism}.

\subsection{The \texorpdfstring{$\Aq$}{two-vertex Type A} case}

The main result of this section is to show that $\Ap \iso \Aq$ if and only if $p=q^{\pm 1}$.

\begin{lemma}\label{lem.Aisos}
Let $q \in \kk^\times$. 

(1) The linear map $\phi:\Aq \to \Aqinv$ defined by $\phi(e_1)=e_1$, $\phi(e_2)=e_2$, and
\begin{align*}
\phi(a)&=c, & \phi(b)&=d, & \phi(c)&=a, & \phi(d)&=b,
\end{align*}
extends to an isomorphism.

(2) The linear map $\psi: \Aq \to \Aqinv$ defined by $\psi(e_1)=e_2$, $\psi(e_2)=e_1$, and
\begin{align*}
 \psi(a)&=qb, & \psi(b)&=a, & \psi(c)&=d, & \psi(d)&=c,
 \end{align*}
extends to an isomorphism.
\end{lemma}
\begin{proof}
Both maps are clearly bijective. To show that they extend to homomorphisms, it is necessary only to show that they respect the relations. For $\phi$,
\begin{align*}
    \phi(a)\phi(b)-\phi(c)\phi(d) &= cd-ab= -(ab-cd)=0,\\
    \phi(b)\phi(a)-q\phi(d)\phi(c) &= dc-qba=q(ba-q^{-1}dc)=0.
\end{align*}
For $\psi$,
\begin{align*}
    \psi(a)\psi(b)-\psi(c)\psi(d) &= qba-dc = q(ba-q\inv dc) = 0, \\
    \psi(b)\psi(a)-q\psi(d)\psi(c) &= qab-qcd= q(ab-cd)=0.
\end{align*}
This proves the result.
\end{proof}

\begin{lemma}\label{lem.typeI}
Let $p,q \in \kk^\times$ with $q \neq 1$.
Assume that $\Phi: \Aq\to \Ap$ is a graded isomorphism of type I.
Then $\Phi$ is given by 
\begin{align}\label{eq.Phi}
\Phi(a)=k_1a+k_2c, \qquad
\Phi(b)=\ell_1b+ell_2d, \qquad
\Phi(c)=k_3a+k_4c, \qquad
\Phi(d)=\ell_3b+\ell_4d,
\end{align}
where $k_1k_4-k_2k_3 \neq 0$ and $\ell_1\ell_4 - \ell_2\ell_3 \neq 0$. If $k_1,k_4 \neq 0$, then $p=q$.
\end{lemma}
\begin{proof}
The relations give
\begin{align*}
0   &=\Phi(ab-cd) 
    %=(k_1a+k_2c)(\ell_1b+\ell_2d)-(k_3a+k_4c)(\ell_3b+\ell_4d) \\     
    = cd(k_1\ell_1+k_2\ell_2-k_3\ell_3-k_4\ell_4)+ad(k_1\ell_2-k_3\ell_4)+cb(k_2\ell_1-k_4\ell_3), \\
0 &= \Phi(ba-qdc)=dc(pk_1\ell_1+k_2\ell_2-qpk_3\ell_3-qk_4\ell_4)+bc(k_2\ell_1-qk_4\ell_3)+da(k_1\ell_2-qk_3\ell_4).
\end{align*}
By considering the coefficients of $ad, cb$ in the first equation and $bc,da$ in the second,

\vspace*{-.7\multicolsep}

\begin{center}
\noindent
\begin{minipage}{0.2\textwidth}
    \begin{equation}
        \label{eq.Aq1}    k_1\ell_2=k_3\ell_4,
    \end{equation} 
\end{minipage}%
\begin{minipage}{0.2\textwidth}
\end{minipage}%
\begin{minipage}{0.2\textwidth}
    \begin{equation}
        \label{eq.Aq2}    k_2\ell_1=k_4\ell_3,
    \end{equation}
\end{minipage}
\begin{minipage}{0.2\textwidth}
\end{minipage}%
\begin{minipage}{0.2\textwidth}
    \begin{equation}
        \label{eq.Aq3}    k_1\ell_2 = qk_3\ell_4,
    \end{equation}
\end{minipage}
\begin{minipage}{0.2\textwidth}
\end{minipage}%
\begin{minipage}{0.2\textwidth}
    \begin{equation}
        \label{eq.Aq4}    k_2\ell_1=qk_4\ell_3.
    \end{equation}
\end{minipage}
\end{center}

\vskip1em

Since $q \neq 1$, then \eqref{eq.Aq1} and \eqref{eq.Aq3} cannot both simultaneously be true, unless both sides are zero.  Since $k_1 \neq 0$, then $\ell_2=0$. Then $\ell_4 \neq 0$ by bijectivity, so $k_3=0$. A similar argument using \eqref{eq.Aq2} and \eqref{eq.Aq4} shows that $\ell_3=k_2=0$. Now the coefficient of $cd$ gives $k_1\ell_1=k_4\ell_4$ and the coefficient of $dc$ gives $pk_1\ell_1=qk_4\ell_4$, so $p=q$.
\end{proof}

\begin{proposition}\label{prop.A2}
Let $p,q\in \kk^\times$. Then $\Aq \iso \Ap$ if and only if $p = q^{\pm 1}$.
\end{proposition}
\begin{proof}
If $p=q$ then the result is obvious. If $p=q\inv$, then $\Aq \iso \Ap$ by Lemma \ref{lem.Aisos}.

Now let $\Phi:\Aq \to \Ap$ be a (graded) isomorphism and $q \neq 1$. Suppose $\Phi$ is type I. Then $\Phi$ is given by \eqref{eq.Phi} with $k_1k_4-k_2k_3 \neq 0$ and $\ell_1\ell_4 - \ell_2\ell_3 \neq 0$. If $k_1,k_4 \neq 0$, then by Lemma \ref{lem.typeI}, $p=q$.

On the other hand, suppose $k_1=0$ or $k_4=0$. This implies that $k_2,k_3 \neq 0$. If $\phi$ is as in Lemma \ref{lem.Aisos}, then $\pi \circ \Phi:\Aq \to A_{p\inv}$ is an isomorphism of type I and Lemma \ref{lem.typeI} implies that $p\inv = q$.

Now suppose that $\Phi$ is type II. If $\psi$ is as in Lemma \ref{lem.Aisos}, then $\psi \circ \Phi$ is type I. Hence, the argument in the previous paragraph shows that $p = q^{\pm 1}$.

Finally, suppose $q=1$ and consider $\Phi\inv:\Ap \to \Aq$. If $p \neq 1$, then the argument above shows that $1 = q = p^{\pm 1} \neq 1$, a contradiction. Hence, the result holds for all $p,q \in \kk^\times$.
\end{proof}

\subsection{The \texorpdfstring{$\Dq$}{Type D} case}
This section establishes the corresponding isomorphism problem for the $\Dq$.

\begin{lemma}\label{lem.D}
Let $q \in \kk^\times$. 
\begin{enumerate}
\item The linear map $\phi:\Dq \to \Dqinv$ defined by $\phi(e_i)=e_1$, $\phi(e_2)=e_2$, and 
\begin{align*}
\phi(a)&=qa + (q-1)c, & \phi(b)&=qb, & \phi(c)&=c, & \phi(d)&=(q-1)b+d,
\end{align*}
extends to an isomorphism.

\item The linear map $\psi:\Dq \to \Dq$ defined by $\psi(e_1)=e_2$, $\psi(e_2)=e_1$, and
\begin{align*}
\psi(a)&=qd, & \psi(b)&=a+c, & \psi(c)&=b, & \psi(d)&=a,
\end{align*}
extends to an isomorphism.
\end{enumerate}
\end{lemma}
\begin{proof}
Both maps are clearly bijective. That $\phi$ respects the relations is established below:
\begin{align*}
\phi(b)\phi(a)-q\phi(d)\phi(c)
    &= (qb)(qa + (q-1)c) - q((q-1)b+d)c \\
    &= q^2ba + q(q-1)bc - q(q-1)bc - qdc \\
    &= q^2(ba-q\inv dc) = 0, \\
\phi(a)\phi(b) - (\phi(a)+\phi(c))\phi(d)
    &= (qa + (q-1)c)(qb) - q(a+c)((q-1)b+d) \\
    &= q^2ab + q(q-1)cb - q(q-1)ab - q(q-1)cb - q(a+c)d \\
    &= q(ab - (a+c)d)=0.
\end{align*}
The relations for $\psi$ are checked below:
\begin{align*}
\psi(b)\psi(a)-q\psi(d)\psi(c)
    &= (a+c)(qd) - qab = -q(ab - (a+c)d) = 0, \\    
\psi(a)\psi(b) - (\psi(a)+\psi(c))\psi(d)
    &= (qd)(a+c) - (qd+b)a
%    = qda + qdc - qda - ba
    = -(ba-qdc) = 0.
\end{align*}
This proves the result.
\end{proof}

\begin{proposition}\label{prop.D}
Let $p,q \in \kk^\times$. Then $\Dq \iso \Dp$ if and only if $p=q^{\pm 1}$.
\end{proposition}
\begin{proof}
If $p=q^{\pm 1}$, then $\Dq \iso \Dp$ by Lemma \ref{lem.D}. 

Let $\Phi:\Dq \to \Dp$ be an isomorphism. After possibly composing with $\psi$ as in Lemma \ref{lem.D}, $\Phi$ may be assumed to be type I. Thus, there are scalars $k_i,\ell_i \in \kk$ satisfying $k_1k_4-k_2k_3 \neq 0$ and $\ell_1\ell_4-\ell_2\ell_3 \neq 0$ such that
\begin{align*}
\Phi(a)=k_1a+k_2c, \qquad
\Phi(b)=\ell_1b+\ell_2d, \qquad
\Phi(c)=k_3a+k_4c, \qquad
\Phi(d)=\ell_3b+\ell_4d.
\end{align*}
Then
\begin{align*}
    0 &=\Phi(ab-(a+c)d)
    = ad(k_1\ell_1+k_1\ell_2-k_1\ell_3-k_1\ell_4-k_3\ell_3-k_3\ell_4) \\
    &\qquad + cb(k_2\ell_1-k_2\ell_3-k_4\ell_3)+ cd(k_1\ell_1+k_2\ell_2-k_1\ell_3-k_2\ell_4-k_3\ell_3-k_4\ell_4), \\
    0 &= \Phi(ba-qdc) = bc(k_2\ell_1-qk_4\ell_3)+da(k_1\ell_2-qk_3\ell_4)+dc(pk_1\ell_1+k_2\ell_2-pqk_3\ell_3-qk_4\ell_4).
\end{align*}
This implies the following equations:
\begin{align}
\label{eq.D1} k_2\ell_1 &= k_2\ell_3+k_4\ell_3 \\
\label{eq.D2} k_1\ell_2+k_1\ell_1 
    &= k_1\ell_4+k_3\ell_4+k_1\ell_3+k_3\ell_3 \\
\label{eq.D3}    k_2\ell_2+k_1\ell_1 
    &= k_2\ell_4+k_4\ell_4+k_1\ell_3+k_3\ell_3 \\
\label{eq.D4}    k_2\ell_1&=qk_4\ell_3 \\
\label{eq.D5}    k_1\ell_2&=qk_3\ell_4 \\
\label{eq.D6}    k_2\ell_2+pk_1\ell_1
    &= qk_4\ell_4+pqk_3\ell_3.
\end{align}

Suppose $\ell_3=0$. Then, by \eqref{eq.D4}, $k_2\ell_1=0$. If $\ell_1=0$ then $\Phi$ is not injective, so $k_2=0$. So by \eqref{eq.D6}, $pk_1\ell_1=qk_4\ell_4$, and \eqref{eq.D3} gives that $k_1\ell_1=k_4\ell_4$, so $p=q$. Henceforth, assume that $\ell_3 \neq 0$.

From \eqref{eq.D1} and \eqref{eq.D4}, $(q-1)k_4\ell_3=k_2\ell_3$, so 
\begin{align}\label{eq.D7}
    (q-1)k_4=k_2.
\end{align}
If $q=1$, then $k_2=0$ and from \eqref{eq.D4}, $k_4\ell_3=0$, so $k_4=0$ violating surjectivity. Henceforth, assume $q \neq 1$.

Suppose $k_1=0$. By \eqref{eq.D5}, $qk_3\ell_4=0$. If $k_3=0$ then $\Phi$ is not surjective, so $\ell_4=0$. But then \eqref{eq.D2} gives $k_3\ell_3=0$, a contradiction, so $k_1\neq0$. A symmetric argument gives that $\ell_4 \neq 0$. Thus, \eqref{eq.D5} gives 
\begin{align}\label{eq.D8}
    \ell_2 = q\frac{k_3}{k_1}\ell_4.
\end{align}

Taking the difference of \eqref{eq.D2} and \eqref{eq.D3} gives
\begin{align*}
    (k_1-k_2)\ell_2 &= (k_1-k_2+k_3-k_4)\ell_4 \\
    (k_1-k_2)q\frac{k_3}{k_1}\ell_4 &= (k_1-k_2+k_3-k_4)\ell_4 
        \qquad \text{by \eqref{eq.D8}} \\ 
    (k_1-k_2)qk_3 &= (k_1-k_2+k_3-k_4)k_1 \\
    0 &= k_1^2 - k_1(k_2+k_4) + (1-q)k_1k_3 + qk_2k_3 \\
    0 &= k_1^2 - qk_1k_4 + (1-q)k_1k_3 + q(q-1)k_4k_3 
        \qquad \text{by \eqref{eq.D7}} \\
    0 &= (k_1 - qk_4)(k_1 - (q-1)k_3).
\end{align*}
If $k_1=(q-1)k_3$, then $k_1k_4-k_2k_3 = 0$. Thus,
\begin{align}\label{eq.D9}
    k_1=qk_4.
\end{align}

By \eqref{eq.D4} and \eqref{eq.D7}, $(q-1)k_4\ell_1=qk_4\ell_3$. If $k_4=0$, then $k_2=0$ by \eqref{eq.D7}, violating surjectivity. Thus, 
\begin{align}\label{eq.D9a}
    \ell_3=(1-q\inv)\ell_1. 
\end{align}
Plugging this and \eqref{eq.D7} into \eqref{eq.D3} gives 
\begin{align}\label{eq.D10}
    k_2\ell_2+q\inv k_1\ell_1=qk_4\ell_4+k_3\ell_3. 
\end{align}
Subtracting \eqref{eq.D10} that from \eqref{eq.D6} gives 
$(p-q\inv)k_1\ell_1 =(pq-1)k_3\ell_3$.
This $p=q\inv$, then the proof is complete. Assume $p \neq q\inv$, then $q\inv k_1\ell_1 =k_3\ell_3$. By \eqref{eq.D9} and \eqref{eq.D9a}, $k_4\ell_1 = k_3(1-q\inv )\ell_1$. If $\ell_1=0$, then \eqref{eq.D9a} gives $\ell_3=0$, violating surjectivity. Thus,
\begin{align}\label{eq.D12}
    k_3 &=\frac{q}{q-1}k_4.
\end{align}

%Suppose $pq=1$. From \eqref{eq.D2}, \eqref{eq.D9}, and \eqref{eq.D9a} we have
%\begin{align*}
%k_1\ell_2+k_1\ell_1 &= k_1\ell_4+k_3\ell_4+k_1\ell_3+k_3\ell_3
%k_1\ell_2+k_1\ell_1 &= qk_4\ell_4+k_3\ell_4+(1-q\inv)k_1\ell_1+k_3\ell_3 \\
%k_1\ell_2 + p k_1\ell_1 &= qk_4\ell_4+k_3\ell_4+k_3\ell_3.
%\end{align*}
%Subtracting \eqref{eq.D6} gives $k_3\ell_4=0$. If $\ell_4=0$, then $\ell_2=0$ by \eqref{eq.D8}, violating surjectivity. If $k_3=0$, then $k_1$

%\begin{align*}
%q\inv k_1\ell_1 &=k_3\ell_3 \\
%q\inv k_1\ell_1 &= k_3(1-q\inv )\ell_1 \quad\text{by \eqref{eq.D9a}} \\
%q\inv k_1 &=(1-q\inv)k_3 \\
%k_4 &=(1-q\inv)k_3 \\
%k_3 &=\frac{q}{q-1}k_4
%\end{align*}

Now, note that 
\[ \begin{vmatrix}k_1 & k_2 \\ k_3 & k_4\end{vmatrix}=\begin{vmatrix}qk_4 & (q-1)k_4 \\ \frac{q}{q-1}k_4 & k_4\end{vmatrix}=k_4\begin{vmatrix}q & (q-1) \\ \frac{q}{q-1} & 1\end{vmatrix}=0,\]
violating bijectivity of $\Phi$. Thus, if $\ell_3 \neq 0$, then $p=q\inv$.
\end{proof}

\subsection{The classification}

Combining the results above now gives the classification referenced at the beginning on this section.

\begin{theorem}\label{thm.class}
For $q \in \kk^\times$, the algebras $\Aq$, $\Bq$, $\algJ$, and $\Dq$ are twisted graded Calabi--Yau. If $\cA$ is an algebra satisfying Hypothesis \ref{hyp.CY2}, then $\cA$ is isomorphic to one of the $\Aq$, $\Bq$, $\algJ$, or $\Dq$. Moreover, these algebras are pairwise non-isomorphic except $\Aq \iso \Aqinv$ and $\Dq \iso \Dqinv$.
\end{theorem}
\begin{proof}
By Lemmas \ref{lem.AnCY}, \ref{lem.BnCY}, \ref{lem.CCY}, and \ref{lem.DCY}, the given algebras are all twisted Calabi--Yau.

By Lemmas \ref{lem.2vert}, \ref{lem.M1}, and \ref{lem.M2},  every $\cA$ satisfying Hypothesis \ref{hyp.CY2} belongs to one of the four families: $\Aq$, $\Bq$, $\algJ$, $\Dq$. The algebras $\Aq,\Dq$ are nonisomorphic to $\Bq,\algJ$ because they live of distinct quivers. The proof of Lemma \ref{lem.M1} shows that $\Aq \not\iso \Dp$ for any $p,q \in \kk^\times$ and Lemma \ref{lem.M2} shows $\Bq \not\iso \algJ$ because the Nakayama automorphisms are different.

By Proposition \ref{prop.A2}, $\Aq \iso \Ap$ if and only if $p=q^{\pm 1}$. By Corollary \ref{cor.B2}, $\Bq \iso \Bp$ if and only if $p=q$. By Proposition \ref{prop.D}, $\Dp \iso \Dq$ if and only if $p=q^{\pm 1}$.
\end{proof}

%\bibliographystyle{myplain}
%\bibliography{biblio}

\end{document}